\documentclass[a4paper,10pt]{amsart}
\usepackage{amsmath,amsthm,amssymb,amsfonts,enumerate,color,bbm}

\newcommand{\abs}[1]{\left\vert#1\right\vert}
\newcommand{\set}[1]{\left\{#1\right\}}
\newcommand{\Real}{\mathbb{R}}
\renewcommand{\a}{\alpha}
\renewcommand{\b}{\beta}
\newcommand{\s}{\sigma}
\renewcommand{\t}{\theta}
\renewcommand{\v}{\varphi}

\renewcommand{\P}{\mathcal{P}}
\newcommand{\J}{\mathcal{J}}
\newcommand{\K}{\mathcal{K}}
\renewcommand{\S}{\mathbb{S}}

\newcommand{\CP}{P_l(\mathbb{C})}
\newcommand{\HP}{P_l(\mathbb{H})}
\newcommand{\Ca}{P_2(\mathbb{C}\mathbbm{a}\mathbbm{y})}
\newcommand{\B}{\mathbb{B}}

\newtheorem{thm}{Theorem}[section]

\newtheorem{lem}[thm]{Lemma}

\theoremstyle{definition}

\newtheorem{rem}[thm]{Remark}

\allowdisplaybreaks

\numberwithin{equation}{section}

\author[\'O. Ciaurri, L. Roncal, and P. R. Stinga]{\'Oscar Ciaurri \and Luz Roncal \and Pablo Ra\'ul Stinga}
\address{Departamento de Matem\'aticas y Computaci\'on\\
         Universidad de La Rioja\\
         26004 Logro\~no, Spain}
\email{oscar.ciaurri@unirioja.es, luz.roncal@unirioja.es}

\address{Department of Mathematics\\
	The University of Texas at Austin\\
	1 University Station C1200\\
	78712-1202 Austin, TX\\
	United States of America}
\email{stinga@math.utexas.edu}

\thanks{The first and second authors were partially supported by grant MTM2012-36732-C03-02 from Spanish Government. This research was developed while the third author was a researcher at Departmento de Matem\'aticas y Computaci\'on of Universidad de La Rioja, Spain, under grant COLABORA 2010/01 from Planes Riojanos de I+D+I. The third author was partially supported by grant MTM2011-28149-C02-01 from Spanish Government}



\begin{document}

\title[Fractional integrals on Riemannian symmetric spaces]{Fractional integrals on compact \\ Riemannian symmetric spaces of rank one}

\begin{abstract}
In this paper we study mixed norm boundedness for fractional integrals related to Laplace--Beltrami operators on compact Riemannian symmetric spaces of rank one. The key point is the analysis of weighted inequalities for fractional integral operators associated to trigonometric Jacobi polynomials expansions. In particular, we find a novel sharp estimate for the Jacobi fractional integral kernel with explicit dependence on the type parameters.
\end{abstract}

\maketitle

\section{Introduction}

Let $M$ be a Riemannian symmetric space of compact type and rank one or, what is the same thing, a compact two-point homogeneous space. According to H.-C. Wang \cite{Wang}, all the compact two-point homogeneous spaces are
\begin{enumerate}
    \item The sphere $\S^d\subset\Real^{d+1}$, $d\geq1$,
    \item The real projective space $P_d(\Real)$, $d\geq2$,
    \item The complex projective space $\CP$, $l\geq2$,
    \item The quaternionic projective space $\HP$, $l\geq2$,
    \item The Cayley plane $\Ca$.
\end{enumerate}

Fourier Analysis on Riemannian symmetric spaces has been studied by many authors. Let us briefly mention some of them. For the non-compact case, S. Helgason \cite{Helgason-Duality, Helgason-book, Helgason-libro, Helgason-Bulletin} found the analogous of classical Fourier Analysis. In \cite{Anker-Annals} J.-P. Anker gave results on $L^p$ multipliers. Heat kernel and Green function estimates were found by Anker and L. Ji in \cite{Anker-Ji}. In the compact case, the Radon transform was studied by Helgason in \cite{Helgason-Acta}. In \cite{Sherman-Bulletin, Sherman}, T. O. Sherman developed the counterpart of Helgason's theory for compact type spaces. More recently, S. Thangavelu considered in \cite{Thangavelu} holomorphic Sobolev spaces in compact symmetric spaces, while in \cite{OS} G. \'Olafsson and H. Schlichtkrull presented a local Paley--Wiener theorem. The Harmonic Analysis for compact Lie groups can be found in the monograph by E. M. Stein \cite{Stein}.

On the other hand, Lebesgue mixed norm spaces $L^p(L^r)$ are of interest in Harmonic Analysis and Partial Differential Equations. These spaces were first systematically studied by A. Benedek and R. Panzone in \cite{Benedek-Panzone}. They also considered boundedness of some classical operators (one of them being the fractional integral, which is central to understand Sobolev spaces). In the context of Partial Differential Equations we can mention the remarkable paper by M. Keel and T. Tao \cite{Keel-Tao} on Strichartz estimates for the Schr\"odinger and wave equations. Strichartz estimates refer to boundedness in spaces $L^p_tL^q_x$ ($t$ for time and $x$ for space) and are useful for proving well-posedness of nonlinear equations.

Since a point $x\in\Real^d$ can be written in polar coordinates as $x=r\theta$, for $r>0$ and $x\in\S^{d-1}$, one can regard a function on $\Real^d$ as a function in the two variables: radial $r$ and angular $\theta$. In this way the mixed norm space $L^p_{\mathrm{rad}}(L^2_{\mathrm{ang}})$ arises. This special mixed norm space has a crucial role in Harmonic Analysis. For example, consider the problem of almost everywhere convergence of Bochner--Riesz means for the Fourier transform. One is then led to consider the related maximal operator and to establish its $L^p(\Real^d)$-boundedness properties, a quite difficult task. J. L. Rubio de Francia showed in \cite{Rubio-Duke} that such a maximal operator is bounded in $L^p_{\mathrm{rad}}(L^2_{\mathrm{ang}})$, which is a larger space than $L^p(\Real^d)$ when $p\geq2$, and for the sharp range of exponents in all even dimensions, solving in this case the a.e. convergence problem.

In this paper we focus on mixed norm estimates for fractional integrals on compact Riemannian symmetric spaces of rank one. Each of the manifolds $M$ listed in (1)--(5) above admits essentially only one invariant second order differential operator, the Laplace--Beltrami operator $\tilde{\Delta}_M$, see \cite[Chapter~X]{Helgason-book}. With this we can define an associated fractional integral operator $(-\Delta_M)^{-\sigma/2}$. In principle, it is not clear how to handle this operator in $L^p(M)$. By introducing \textit{polar} coordinates on $M$ appropriate mixed norm spaces $L^p(L^2(M))$ can be defined. These spaces can be seen as versions on $M$ of the $L^p_{\mathrm{rad}}(L^2_{\mathrm{ang}})$ spaces used by Rubio de Francia. See Section \ref{Section:Riemannian}. Our aim is then to prove $L^p(L^2(M))-L^q(L^2(M))$ estimates for $(-\Delta_M)^{-\sigma/2}$.

The first two main results are the following.

\begin{thm}[Sphere and real projective space]\label{Thm:esfera}
Let $M=\S^d$ or $P_d(\Real)$, for $d\geq2$. Let $0<\sigma<1$ and $p,q$ be exponents such that
$$\frac{2d}{d+1}<p\le q<\frac{2d}{d-1},\quad\hbox{and}\quad\frac{1}{q}\ge\frac{1}{p}-\frac{\sigma}{d}.$$
Then there exists a constant $C$ such that for any $f\in L^p(L^2(M))$,
$$\|(-\Delta_{M})^{-\sigma/2}f\|_{L^q(L^2(M))}\leq C\|f\|_{L^p(L^2(M))}.$$
\end{thm}

\begin{thm}[Complex and quaternionic projective spaces and Cayley plane]\label{Thm:proyectivos}
Let $M$ be one of the manifolds $\CP$, $\HP$, for $l\geq2$, or $\Ca$. Define $d=2,4,8$ and $m=l-2,2l-3,3$ for $\CP$, $\HP$ and $\Ca$ respectively. Let $0<\sigma<1$ and $p,q$ be exponents such that
$$\max\left\{\frac{2m+2}{m+3/2},\frac{2d}{2d-1/2}\right\}<p\le q<\min\left\{\frac{2m+2}{m+1/2},\frac{2d}{2d-3/2}\right\},$$
and
$$\frac{1}{q}\ge\frac{1}{p}-\min\left\{\frac{\sigma}{2m+2},\frac{\sigma}{2d}\right\}.$$
Then there exists a constant $C$ such that for any $f\in L^p(L^2(M))$,
$$\|(-\Delta_M)^{-\sigma/2}f\|_{L^q(L^2(M))}\leq C\|f\|_{L^p(L^2(M))}.$$
\end{thm}

Our first two main results above will be consequences of the weighted vector valued $L^p(\ell^2)-L^q(\ell^2)$ estimate for Jacobi fractional integrals that we obtain in Theorem \ref{Thm:Lp-Lq} below.

It turns out that, in the corresponding coordinates, the eigenspaces $\mathcal{H}_n(M)$ of the Laplace--Beltrami operator on $M$ can be described in terms of products of spherical harmonics and Jacobi polynomials, see Section \ref{Section:Riemannian}. Moreover, the type parameters of the Jacobi polynomials therein depend on $j$, where $j$ varies in a range depending on $n$ and $d$. Hence, in order to handle the fractional integrals on $M$, it is needed the precise control of the dependence of the constants with respect to the parameters in Jacobi fractional integrals. The latter is done in our third main result, Theorem \ref{Thm:Jacobi kernel} below. Let us quickly introduce the notation. Given parameters $\a,\b>-1$, the trigonometric Jacobi  polynomials differential operator is given by
\begin{equation}
\label{eq:ecDiferencial Jacobi}
\mathcal{J}^{\a,\b}=-\frac{d^2}{d\theta^2}-\frac{\a-\b+(\a+\b+1) \cos\theta}{\sin\theta}\frac{d}{d\theta}+\left(\frac{\a+\b+1}{2}\right)^2.
\end{equation}
This is a symmetric operator on the interval $(0,\pi)$ equipped with the measure
\begin{equation}\label{medida omojenea}
d\mu_{\a,\b}(\theta)=(\sin\tfrac{\theta}{2})^{2\a+1}(\cos\tfrac{\theta}{2})^{2\b+1}d\theta.
\end{equation}
Its spectral resolution is given in terms of trigonometric Jacobi polynomials $\mathcal{P}_n^{(\a,\b)}(\theta)$, see \eqref{trig pol}. The fractional integral, denoted by $(\mathcal{J}^{\a,\b})^{-\sigma/2}$, is defined in the usual way \eqref{definicion obvia}, or equivalently, with a formula involving the Poisson semigroup generated by $\J^{\a,\b}$ as in \eqref{def con semigrupo}.

\begin{thm}[Sharp estimate for Jacobi fractional integral kernel]\label{Thm:Jacobi kernel}
Let $0<\s<1$, $\a\ge-1/2$ and $\b>-1/2$. Denote by $\K^{\a,\b}_\sigma(\t,\v)$, for $\t,\v\in(0,\pi)$, the kernel of the fractional integral $(\J^{\a,\b})^{-\s/2}$, see \eqref{Jacobi frac int kernel}. Then there exists a constant $C_\s$ that only depends on $\s$ and not on $\a$ and $\b$, such that
$$0\leq\K^{\a,\b}_\s(\t,\v)\le \frac{C_{\s}}{\left(\sin\frac{\t}{2}\sin\frac{\v}{2}\right)^{\a+1/2} \left(\cos\frac{\t}{2}\cos\frac{\v}{2}\right)^{\b+1/2}}\Bigg(\frac{1}{\b+1/2}+\frac{1}{|\sin\frac{\t-\v}{4}|^{1-\s}}\Bigg).$$
\end{thm}

Theorem \ref{Thm:Jacobi kernel} allows us to prove

\begin{thm}\label{Thm:Lp-Lq}
Let $0<\sigma<1$, $\a\geq-1/2$, $\b>-1/2$ and $0\leq a,b<\infty$. Define $$u_j(\t)=(\sin\tfrac{\t}{2})^{aj}(\cos\tfrac{\t}{2})^{bj},\quad\t\in(0,\pi),~j=0,1,\ldots.$$
Then there exists a constant $C$ depending only on $\sigma$, $\a$ and $\b$, such that
$$\Big\|\Big(\sum_{j=0}^\infty|u_j(\mathcal{J}^{\a+aj,\b+bj})^{-\s/2}(u_j^{-1}f_j)|^2\Big)^{1/2}\Big\|_{L^q((0,\pi),d\mu_{\a,\b})}\le C\Big\|\Big(\sum_{j=0}^\infty |f_j|^2\Big)^{1/2}\Big\|_{L^p((0,\pi),d\mu_{\a,\b})},$$
for all $f_j\in L^p((0,\pi),d\mu_{\a,\b})$, provided that the exponents $p,q$ satisfy
\begin{equation}\label{chu 1}
1\le \max\left\{\frac{2\a+2}{\a+3/2},\frac{2\b+2}{\b+3/2}\right\}<p\le q<\min\left\{\frac{2\a+2}{\a+1/2},\frac{2\b+2}{\b+1/2}\right\}\le\infty
\end{equation}
and
\begin{equation}\label{chu 2}
\frac{1}{q}\ge\frac{1}{p}-\min\left\{\frac{\sigma}{2\a+2},\frac{\sigma}{2\b+2}\right\}.
\end{equation}
\end{thm}

A notion of fractional integration for ultraspherical expansions (which corresponds to $\a=\b=\lambda-1/2$ in the Jacobi case) was first introduced and analyzed by B. Muckenhoupt and E. M. Stein in the seminal work \cite{Muckenhoupt-Stein}. Our definition is slightly different from theirs because we use the multiplier $(n+\lambda)^{-\s}$ while they consider $n^{-\s}$. On the other hand, weighted inequalities for fractional integrals with respect to certain Jacobi expansions were studied by B. Muckenhoupt in \cite{Muckenhoupt}. His results are not enough for our purposes here because in \cite{Muckenhoupt} there is no precise control on the constants with respect to $\a$ and $\b$. We point out that in \cite{Nowak-Sjogren Calderon} A. Nowak and P. Sj\"ogren considered several Calder\'on--Zygmund operators related to the operator \eqref{eq:ecDiferencial Jacobi}. The fractional integral was not analyzed there. Moreover, it is not clear from the estimates in \cite{Nowak-Sjogren Calderon} how to track down the exact dependence of the constants on $\a$ and $\b$.

The paper is organized as follows. In Section \ref{Section:Kernel} we introduce the Jacobi fractional integral and give the proofs of Theorem \ref{Thm:Jacobi kernel} and Theorem \ref{Thm:Lp-Lq}. The results on fractional integrals for compact Riemannian symmetric spaces of rank one are proved in Section \ref{Section:Riemannian}.

\section{Fractional integrals for Jacobi expansions}\label{Section:Kernel}

The standard Jacobi polynomials of degree $n\geq0$ and type $\a,\b>-1$ are given by
\begin{equation}\label{eq1}
P_n^{(\a,\b)}(x)=(1-x)^{-\a}(1+x)^{-\b}\frac{(-1)^n}{2^nn!}\left(\frac{d}{dx}\right)^n\set{(1-x)^{n+\a}(1+x)^{n+\b}},\quad x\in(-1,1),
\end{equation}
see \cite[(4.3.1)]{Szego}. These form an orthogonal basis of $L^2((-1,1),(1-x)^\a(1+x)^\b dx)$. After making the change of variable $x=\cos\theta$, we obtain the normalized Jacobi trigonometric polynomials
\begin{equation}\label{trig pol}
\mathcal{P}_n^{(\a,\b)}(\theta)=d_n^{\a,\b}P_n^{(\a,\b)}(\cos\theta),
\end{equation}
where the normalizing factor is
\begin{equation}\label{que}
\begin{aligned}
    d_n^{\a,\b} &= 2^{\frac{\a+\b+1}{2}}\|P_n^{(\a,\b)}\|_{L^2((-1,1),(1-x)^\a(1+x)^\b dx)}^{-1} \\
     &= \left(\frac{(2n+\a+\b+1)\Gamma(n+1)\Gamma(n+\a+\b+1)}{\Gamma(n+\a+1)\Gamma(n+\b+1)}\right)^{1/2}.
\end{aligned}
\end{equation}
The trigonometric polynomials $\P_n^{(\a,\b)}$ are eigenfunctions of the differential operator \eqref{eq:ecDiferencial Jacobi}. Indeed, we have $\J^{\a,\b}\P_n^{(\a,\b)}=\lambda_n^{\a,\b}\P_n^{(\a,\b)}$, with eigenvalue
$$\lambda_n^{\a,\b}=\left(n+\tfrac{\a+\b+1}{2}\right)^2.$$
Moreover, the system $\{\mathcal{P}_n^{(\a,\b)}\}_{n\ge0}$ forms a complete orthonormal basis of $$L^2(d\mu_{\a,\b}):=L^2((0,\pi),d\mu_{\a,\b}(\t)),$$
with $d\mu_{\a,\b}$ as in \eqref{medida omojenea}. For further references about Jacobi polynomials, see \cite[Chapter IV]{Szego}.

\subsection{The Jacobi-Poisson kernel}

The Poisson semigroup related to $\J^{\a,\b}$ is initially defined in $L^2(d\mu_{\a,\b})$ as
$$\P_t^{\a,\b}f(\t)=\sum_{n=0}^\infty e^{-t\left|n+\frac{\a+\b+1}{2}\right|}c_n^{\a,\b}(f)\mathcal{P}_n^{(\a,\b)}(\theta),\quad t>0,$$
where the Fourier-Jacobi coefficient is
$$c_n^{\a,\b}(f)=\int_0^{\pi}f(\t)\mathcal{P}_n^{(\a,\b)}(\t)\,d\mu_{\a,\b}(\t).$$
We can write the Poisson semigroup $\{\P_t^{\a,\b}\}_{t>0}$ as an integral operator
$$\P_t^{\a,\b}f(\t)=\int_0^{\pi}\P_t^{\a,\b}(\t,\v)f(\v)\,d\mu_{\a,\b}(\v).$$
The Jacobi-Poisson kernel is given by
\begin{equation}\label{serie}
\P_t^{\a,\b}(\t,\v)=\sum_{n=0}^{\infty}e^{-t\left|n+\frac{\a+\b+1}{2}\right|}\mathcal{P}_n^{(\a,\b)}(\theta)\mathcal{P}_n^{(\a,\b)}(\v).
\end{equation}
We will need a more or less explicit expression of this kernel when $\a,\b\geq-1/2$. The formula for it was obtained in \cite{Nowak-Sjogren Calderon}. For the sake of completeness we sketch the computations here. Let
\begin{equation}\label{zeta}
z\equiv z(u,v,\t,\v):=u\sin\tfrac{\t}{2}\sin\tfrac{\v}{2}+v\cos\tfrac{\t}{2}\cos\tfrac{\v}{2}, \quad \t,\v\in(0,\pi),~u,v\in[-1,1].
\end{equation}
Consider on $[-1,1]$ the measure
\begin{equation}\label{medida}
d\Pi_\a(u)=\frac{\Gamma(\a+1)}{\sqrt{\pi}\Gamma(\a+1/2)}(1-u^2)^{\a-1/2}\,du,\quad\a>-1/2,
\end{equation}
(the analogous definition for $d\Pi_\b(v)$). In the limit case $\a=-1/2$, we put $\Pi_{-1/2}=\frac12(\delta_{-1}+\delta_1)$. The expression for the Jacobi-Poisson kernel can be achieved by using the product formula for Jacobi polynomials found in \cite[Theorem~(2.2)]{Dijksma-Koornwinder}:
$$\mathcal{P}_n^{(\a,\b)}(\t)\mathcal{P}_n^{(\a,\b)}(\v)= \frac{(2n+\a+\b+1)\Gamma(\a+\b+1)}{\Gamma(\a+1)\Gamma(\b+1)}\int_{-1}^1 \int_{-1}^1C_{2n}^{\a+\b+1}(z(u,v,\t,\v))\,d\Pi_\a(u)\,d\Pi_\b(v),$$
valid for $\a,\b>-1/2$. Here,
\begin{equation}\label{Gegenbauer}
C_k^\lambda(x)=\frac{\Gamma(\lambda+1/2)\Gamma(k+2\lambda)}{\Gamma(2\lambda)\Gamma(k+\lambda+1/2)}P_k^{(\lambda-1/2,\lambda-1/2)}(x)
\end{equation}
are the Gegenbauer polynomials of degree $k$ and parameter $\lambda>-1/2$. Applying the product formula for Jacobi polynomials to \eqref{serie}, we are reduced to compute
\begin{equation}\label{sum}
\sum_{n=0}^\infty (2n+\a+\b+1)e^{-\frac{t}{2}(2n+\a+\b+1)}C_{2n}^{\a+\b+1}(z).
\end{equation}
The generating formula for Gegenbauer polynomials gives
$$\sum_{k=0}^\infty\frac{k+\lambda}{\lambda}C_k^\lambda(z)r^k=\frac{1-r^2}{(1-2rz+r^2)^{\lambda+1}},\quad\abs{r}<1,~\lambda>0.$$
Let $\lambda=\a+\b+1$ and $r=e^{-t/2}$ in the identity above. Since the Gegenbauer polynomials of odd degree are odd functions and those of even degree are even functions, if we take the sum above only for $k$ even then we get the even part of the right-hand side. Hence, the sum in \eqref{sum} equals to
$$\frac{(\a+\b+1)e^{-t\frac{\a+\b+1}{2}}(1-e^{-t})}{2(1-2e^{-t/2}z+e^{-t})^{\a+\b+2}} +\frac{(\a+\b+1)e^{-t\frac{\a+\b+1}{2}}(1-e^{-t})}{2(1+2e^{-t/2}z+e^{-t})^{\a+\b+2}}.$$
Therefore, if we collect all the terms, take into account that $d\Pi_\a(-u)=d\Pi_\a(u)$ and complete the squares in the denominator of the expression above, then we get the desired formula:
\begin{equation}\label{Poisson kernel}
\begin{aligned}
    \P_t^{\a,\b}(\t,\v) &= \frac{\Gamma(\a+\b+2)}{\Gamma(\a+1)\Gamma(\b+1)}\\
    &\quad\times\int_{-1}^1\int_{-1}^1\frac{e^{-t\frac{\a+\b+1}{2}}(1-e^{-t})} {\left((1-e^{-t/2})^2+2e^{-t/2}(1-z(u,v,\t,\v))\right)^{\a+\b+2}}\,d\Pi_\a(u)\,d\Pi_\b(v).
\end{aligned}
\end{equation}
By continuity arguments, the representation for the Jacobi-Poisson kernel in \eqref{Poisson kernel} remains valid in the limiting cases $\a=-1/2$ or $\b=-1/2$.

\subsection{The kernel of the Jacobi fractional integral and proof of Theorem \ref{Thm:Jacobi kernel}}

Given $\s>0$, the Jacobi fractional integral is defined in $L^2(d\mu_{\a,\b})$ as
\begin{equation}\label{definicion obvia}
(\mathcal{J}^{\a,\b})^{-\s/2}f(\theta)=\sum_{n=0}^\infty\frac{1}{\left(n+\frac{\a+\b+1}{2}\right)^{\s}}\,c_n^{\a,\b}(f)\P_n^{(\a,\b)}(\theta).
\end{equation}
It is easy to check that, for $0<\s<1$, the following expression involving the Poisson semigroup gives an equivalent definition:
\begin{equation}\label{def con semigrupo}
(\mathcal{J}^{\a,\b})^{-\s/2}f(\theta)=\frac{1}{\Gamma(\s)}\int_0^\infty\P_t^{\a,\b}f(\theta)\,\frac{dt}{t^{1-\s}}= \int_{0}^{\pi}\K^{\a,\b}_\s(\t,\v)f(\varphi)\,d\mu_{\a,\b}(\varphi),
\end{equation}
where, for $\P_t^{\a,\b}(\t,\v)$ as in \eqref{Poisson kernel},
\begin{equation}\label{Jacobi frac int kernel}
\K^{\a,\b}_\s(\t,\v)=\frac{1}{\Gamma(\s)}\int_0^\infty\P_t^{\a,\b}(\theta,\varphi)\,\frac{dt}{t^{1-\s}}.
\end{equation}
The fractional integral $(\mathcal{J}^{\a,\b})^{-\s/2}$ can also be seen as a Laplace transform type multiplier. For general theory on Laplace transform multipliers in compact Lie groups see Stein \cite[Chapter~II]{Stein}.

In order to prove Theorem \ref{Thm:Jacobi kernel} we need the following auxiliary Lemma.

\begin{lem}\label{lem:Sharp estimate}
Let $\gamma\geq-1/2$, $\lambda>-1/2$, $A>B>0$, and $d\Pi_\gamma(u)$ be as in \eqref{medida}. Then
\begin{equation}\label{1}
\int_{-1}^1\frac{d\Pi_\gamma(u)}{(A-Bu)^{\gamma+\lambda+1}}\leq\frac{2^{\gamma+1/2}\Gamma(\gamma+1)\Gamma(\lambda+1/2)} {\sqrt{\pi}\Gamma(\gamma+\lambda+1)}\,\frac{1}{B^{\gamma+1/2}(A-B)^{\lambda+1/2}}.
\end{equation}
Moreover, for $\gamma>-1/2$ and $\lambda=-1/2$, we have
\begin{equation}\label{2}
\int_{-1}^1\frac{d\Pi_{\gamma}(u)}{(A-Bu)^{\gamma+1/2}}\le\frac{2^{\gamma+1/2}\Gamma(\gamma+1)} {\sqrt{\pi}\Gamma(\gamma+1/2)}\left[\frac{1}{\gamma+1/2}+\log\left(\frac{A}{A-B}\right)\right].
\end{equation}
Note that for the case $\gamma,\lambda=-1/2$, the integral boils down to
$\displaystyle\int_{-1}^1d\Pi_{-1/2}(u)=1$.
\end{lem}

\begin{proof}
We have
\begin{align*}
     \int_{-1}^1\frac{d\Pi_\gamma(u)}{(A-Bu)^{\gamma+\lambda+1}}&\leq\frac{2\Gamma(\gamma+1)}{\sqrt{\pi}\Gamma(\gamma+1/2)} \int_0^1\frac{(1-u^2)^{\gamma-1/2}}{(A-Bu)^{\gamma+\lambda+1}}\,du \\
     &\leq \frac{2^{\gamma+1/2}\Gamma(\gamma+1)}{\sqrt{\pi}\Gamma(\gamma+1/2)} \int_0^1\frac{(1-u)^{\gamma-1/2}}{(A-Bu)^{\gamma+\lambda+1}}\,du \\
     &=\frac{2^{\gamma+1/2}\Gamma(\gamma+1)} {\sqrt{\pi}\Gamma(\gamma+1/2)}\int_0^1\frac{t^{\gamma-1/2}}{(A-B(1-t))^{\gamma+\lambda+1}}\,dt \\
     &= \frac{2^{\gamma+1/2}\Gamma(\gamma+1)}{\sqrt{\pi}\Gamma(\gamma+1/2)}\,\frac{1}{(A-B)^{\gamma+\lambda+1}} \int_0^1\frac{t^{\gamma-1/2}}{(1+\frac{B}{A-B}t)^{\gamma+\lambda+1}}\,dt \\
     &= \frac{2^{\gamma+1/2}\Gamma(\gamma+1)}{\sqrt{\pi}\Gamma(\gamma+1/2)}\,\frac{1}{(A-B)^{\gamma+\lambda+1}} \left(\frac{A-B}{B}\right)^{\gamma+1/2}\int_0^{\frac{B}{A-B}}\frac{s^{\gamma-1/2}}{(1+s)^{\gamma+\lambda+1}}ds \\
     &\leq \frac{2^{\gamma+1/2}\Gamma(\gamma+1)}{\sqrt{\pi}\Gamma(\gamma+1/2)}\,\frac{1}{B^{\gamma+1/2}(A-B)^{\lambda+1/2}}\, \frac{\Gamma(\gamma+1/2)\Gamma(\lambda+1/2)}{\Gamma(\gamma+\lambda+1)} \\
     &=\frac{2^{\gamma+1/2}\Gamma(\gamma+1)\Gamma(\lambda+1/2)} {\sqrt{\pi}\Gamma(\gamma+\lambda+1)}\,\frac{1}{B^{\gamma+1/2}(A-B)^{\lambda+1/2}}.
\end{align*}
For the case $\gamma>-1/2$, $\lambda=-1/2$, by using the change of variable $1-u=\frac{A-B}{B}s$ and integration by parts,
\begin{align*}
    \int_{-1}^1\frac{d\Pi_\gamma(u)}{(A-Bu)^{\gamma+1/2}} &\le \frac{2^{\gamma+1/2}\Gamma(\gamma+1)}{\sqrt{\pi}\Gamma(\gamma+1/2)}\int_{0}^1\frac{(1-u)^{\gamma-1/2}}{(A-Bu)^{\gamma+1/2}}\,du \\
    &= \frac{2^{\gamma+1/2}\Gamma(\gamma+1)}{\sqrt{\pi}\Gamma(\gamma+1/2)B^{\gamma+1/2}} \int_{0}^{\frac{B}{A-B}}\frac{s^{\gamma-1/2}}{(1+s)^{\gamma+1/2}}\,ds\\
    &\le\frac{2^{\gamma+1/2}\Gamma(\gamma+1)}{\sqrt{\pi}\Gamma(\gamma+1/2)B^{\gamma+1/2}} \left[\left.\frac{s^{\gamma+1/2}}{(\gamma+1/2)(1+s)^{\gamma+1/2}}\right|^{\infty}_{0} +\int_{0}^{\frac{B}{A-B}}\frac{s^{\gamma+1/2}}{(1+s)^{\gamma+3/2}}\,ds\right]\\
    &\le\frac{2^{\gamma+1/2}\Gamma(\gamma+1)}{\sqrt{\pi}\Gamma(\gamma+1/2)B^{\gamma+1/2}} \left[\frac{1}{\gamma+1/2}+\log\left(\frac{A}{A-B}\right)\right].
\end{align*}
\end{proof}

\begin{proof}[Proof of Theorem \ref{Thm:Jacobi kernel}]
Taking into account the expression for the Jacobi fractional integral kernel \eqref{Jacobi frac int kernel} and the formula for the Poisson kernel \eqref{Poisson kernel}, after applying Fubini's Theorem, we get
\begin{align*}
    &\K_\s^{\a,\b}(\t,\v) \\
     & =\frac{\Gamma(\a+\b+2)}{\Gamma(\s)\Gamma(\a+1)\Gamma(\b+1)}
     \int_{-1}^1\int_{-1}^1\int_0^{\infty}\frac{e^{-t\frac{\a+\b+1}{2}}(1-e^{-t})} {\left((1-e^{-t/2})^2+2e^{-t/2}(1-z)\right)^{\a+\b+2}}\frac{dt}{t^{1-\s}}d\Pi_\a(u)d\Pi_\b(v)\\
     &=\frac{\Gamma(\a+\b+2)}{\Gamma(\s)2^{\a+\b+1}\Gamma(\a+1)\Gamma(\b+1)}\int_{-1}^1\int_{-1}^1\int_0^\infty  \frac{\sinh\frac{t}{2}}{\left(\cosh\frac{t}{2}-1+(1-z)\right)^{\a+\b+2}}\frac{dt}{t^{1-\s}}d\Pi_\a(u)d\Pi_\b(v),
\end{align*}
where $z=z(u,v,\t,\v)$ is as in \eqref{zeta}. Split the integral in $t$ into two parts:
\[
\int_0^\infty\frac{\sinh\frac{t}{2}}{\left(\cosh\frac{t}{2}-1+(1-z)\right)^{\a+\b+2}}\frac{dt}{t^{1-\s}}=\int_0^1~+\int_1^\infty=:J_1+J_2.
\]

Let us begin with $J_2$. Since $1-\s>0$,
\begin{align*}
    J_2 &\le \int_1^\infty\frac{\sinh\frac{t}{2}}{\left(\cosh\frac{t}{2}-1+(1-z)\right)^{\a+\b+2}}\,dt = \frac{2}{(\a+\b+1)\left(\cosh\frac{1}{2}-1+(1-z)\right)^{\a+\b+1}}\\
    &\le \frac{2}{(\a+\b+1)\left(1-z\right)^{\a+\b+1}}.
\end{align*}
Then,
$$\int_1^\infty\P_t^{\a,\b}(\theta,\varphi)\frac{dt}{t^{1-\s}}\le\frac{2\Gamma(\a+\b+1)}{2^{\a+\b+1} \Gamma(\a+1)\Gamma(\b+1)}\int_{-1}^1\int_{-1}^1\frac{d\Pi_\a(u)\,d\Pi_\b(v)} {\left(1-u\sin\frac{\t}{2}\sin\frac{\v}{2}-v\cos\frac{\t}{2}\cos\frac{\v}{2}\right)^{\a+\b+1}}.$$
Applying \eqref{1} of Lemma \ref{lem:Sharp estimate} to the integral against $d\Pi_\a(u)$ with $\gamma=\a\ge-1/2$, $\lambda=\b>-1/2$, $A=1-v\cos\frac{\t}{2}\cos\frac{\v}{2}$ and $B=\sin\frac{\t}{2}\sin\frac{\v}{2}$, we see that the expression above is bounded by
$$\frac{2\Gamma(\b+1/2)}{2^{\b+1/2}\Gamma(\b+1)\sqrt{\pi}\left(\sin\frac{\t}{2}\sin\frac{\v}{2}\right)^{\a+1/2}} \int_{-1}^1\frac{d\Pi_\b(v)}{\left(1-\sin\frac{\t}{2}\sin\frac{\v}{2}-v\cos\frac{\t}{2}\cos\frac{\v}{2}\right)^{\b+1/2}}.$$
Next we use \eqref{2} of Lemma \ref{lem:Sharp estimate} with $\gamma=\b>-1/2$ to bound the latter by
\begin{equation}\label{b}
\frac{2}{\pi\left(\sin\frac{\t}{2}\sin\frac{\v}{2}\right)^{\a+1/2}\left(\cos\frac{\t}{2}\cos\frac{\v}{2}\right)^{\b+1/2}}\left[\frac{1}{\b+1/2}+ \log\left(\frac{1-\sin\frac{\t}{2}\sin\frac{\v}{2}}{1-\sin\frac{\t}{2}\sin\frac{\v}{2}-\cos\frac{\t}{2}\cos\frac{\v}{2}}\right)\right].
\end{equation}
Observe that, when $w\ge1$, $\log w\le w^{\varepsilon}/\varepsilon$ for any $\varepsilon>0$. By taking $\varepsilon=\frac{1-\s}{2}$, we then see that the $\log$ term into the square brackets above is bounded by
\begin{align}
    \frac{2}{1-\s}\left(\frac{1-\sin\frac{\t}{2}\sin\frac{\v}{2}} {1-\sin\frac{\t}{2}\sin\frac{\v}{2}-\cos\frac{\t}{2}\cos\frac{\v}{2}}\right)^{\frac{1-\s}{2}} &\le\frac{2(1-\s)^{-1}}{\left(1-\sin\frac{\t}{2}\sin\frac{\v}{2}-\cos\frac{\t}{2}\cos\frac{\v}{2}\right)^{\frac{1-\s}{2}}} \nonumber\\
    &=\frac{2^{\frac{1+\s}{2}}(1-\s)^{-1}}{|\sin\frac{\t-\v}{4}|^{1-\s}}.\label{a}
\end{align}

Now we deal with $J_1$. Making the change of variable $\cosh\frac{t}{2}-1=s^2$,
$$J_1=\int_0^{\sqrt{\cosh\frac{1}{2}-1}}\frac{4s}{\left(s^2+(1-z)\right)^{\a+\b+2}}\frac{ds}{\left(2\log(s^2+1+s\sqrt{s^2+2}) \right)^{1-\sigma}}.$$
Note that, for $w\ge0$, $\log(1+w)\ge\frac{w}{1+w}$. Therefore $2\log(s^2+1+s\sqrt{s^2+2})\ge 2\frac{s(s+\sqrt{s^2+2})}{1+s^2+s\sqrt{s^2+2}}$ and this quantity is bounded from below by a universal constant times $s$, for all $s\in(0,\cosh\frac{1}{2}-1)$. Hence,
\begin{align*}
    J_1 &\le C_\s\int_0^{\sqrt{\cosh\frac{1}{2}-1}}\frac{s^\s}{(s^2+(1-z))^{\a+\b+2}}\,ds \\
    &= \frac{C_{\s}}{(1-z)^{\a+\b+2}}\int_0^{\sqrt{\cosh\frac{1}{2}-1}}\frac{s^{\s}}{\left(\frac{s^2}{1-z}+1\right)^{\a+\b+2}}\,ds \\
    &= \frac{C_\s}{2(1-z)^{\a+\b+3/2-\s/2}}\int_0^{\frac{\cosh1/2-1}{1-z}}\frac{r^{\frac{\s-1}{2}}}{(r+1)^{\a+\b+2}}\,dr \\
    &\le\frac{C}{(1-z)^{\a+\b+3/2-\s/2}}\int_0^{\infty}\frac{r^{\frac{\s-1}{2}}}{(r+1)^{\a+\b+2}}\,dr\\
    &=\frac{C}{(1-z)^{\a+\b+3/2-\s/2}}\frac{\Gamma(\s/2+1/2)\Gamma(\a+\b+3/2-\s/2)}{\Gamma(\a+\b+2)},
\end{align*}
where $C$ is a constant depending only on $\s$ but not on $\a$ and $\b$. Using this estimate for $J_1$ and applying \eqref{1} of Lemma \ref{lem:Sharp estimate} to the integral against $d\Pi_\a(u)$ and then again to the integral against $d\Pi_\b(v)$ we obtain
\begin{align}
    &\int_0^1\P_t^{\a,\b}(\t,\v)\,\frac{dt}{t^{1-\s}} \nonumber\\
     &\le \frac{C\Gamma(\s/2+1/2)\Gamma(\a+\b+3/2-\s/2)}{2^{\a+\b+1}\Gamma(\a+1)\Gamma(\b+1)}\int_{-1}^1\int_{-1}^1\frac{d\Pi_\a(u)\,d\Pi_\b(v)} {\left(1-u\sin\frac{\t}{2}\sin\frac{\v}{2}-v\cos\frac{\t}{2}\cos\frac{\v}{2}\right)^{\a+\b+3/2-\s/2}}\nonumber \\
     &\le \frac{C\Gamma(\s/2+1/2)\Gamma(\b+1-\s/2)}{2^{\b+1/2}\Gamma(\b+1)\sqrt{\pi} \left(\sin\frac{\t}{2}\sin\frac{\v}{2}\right)^{\a+1/2}}\int_{-1}^1 \frac{d\Pi_\b(v)}{\left(1-\sin\frac{\t}{2}\sin\frac{\v}{2}-v\cos\frac{\t}{2}\cos\frac{\v}{2}\right)^{\b+1-\s/2}}\nonumber \\
     &\le \frac{C\Gamma(\s/2+1/2)\Gamma(1/2-\s/2)}{\pi\left(\sin\frac{\t}{2}\sin\frac{\v}{2}\right)^{\a+1/2} \left(\cos\frac{\t}{2}\cos\frac{\v}{2}\right)^{\b+1/2}} \frac{1}{\left(1-\sin\frac{\t}{2}\sin\frac{\v}{2}-\cos\frac{\t}{2}\cos\frac{\v}{2}\right)^{\frac{1-\s}{2}}}\nonumber \\
     &= \frac{C_\s}{\left(\sin\frac{\t}{2}\sin\frac{\v}{2}\right)^{\a+1/2} \left(\cos\frac{\t}{2}\cos\frac{\v}{2}\right)^{\b+1/2}} \frac{1}{|\sin\frac{\t-\v}{4}|^{1-\s}}.\nonumber
\end{align}
Plugging together this last estimate with \eqref{b} and \eqref{a} the conclusion follows.
\end{proof}

\subsection{Proof of Theorem \ref{Thm:Lp-Lq}}

To prove Theorem \ref{Thm:Lp-Lq} we first need an analysis of weighted inequalities of fractional integral operators on the interval $(0,\pi)$.

\subsubsection{Weighted inequalities for fractional integrals on $(0,\pi)$}

For $0<\sigma<1$, we denote by $\mathcal{I}_{\sigma}$ the fractional integral operator defined as
$$\mathcal{I}_{\s}g(\t)=\int_0^{\pi}\frac{g(\v)}{|\t-\v|^{1-\sigma}}\,d\v, \quad\hbox{for}~g\in L^1((0,\pi),d\t).$$
A weight $w$ is a strictly positive and finite almost everywhere function on $(0,\pi)$. Recall that $A_\infty=\bigcup_{1\leq p<\infty}A_p$, where the Muckenhoupt class $A_p$, $p>1$, is the set of locally integrable weights $w$ for which $w^{-1/(p-1)}\in L^1_{\mathrm{loc}}(0,\pi)$ and
$$\sup_{I\subset(0,\pi)}\left(\frac{1}{|I|}\int_Iw(\t)\,d\t\right)\left(\frac{1}{|I|}\int_Iw(\t)^{-1/(p-1)}\,d\t\right)^{p-1}<\infty.$$
Here $|I|$ denotes the length of the interval $I$. See \cite[Chapter~7]{Duo}. The following result on boundedness of fractional integrals can be easily deduced from the paper by A. Bernardis and O. Salinas \cite{Bernardis-Salinas}.

\begin{thm}\label{Thm:BS}
Suppose $0<\s<1$ and $1<p\le q<\infty$. Let $(w,v)$ be a pair of weights with $w\in A_\infty$ and $\bar{v}:=v^{-1/(p-1)}\in A_{\infty}$. Then
$$\|\mathcal{I}_{\s}f\|_{L^q((0,\pi),w\,d\t)}\le C\|f\|_{L^p((0,\pi),v\,d\t)},\quad\hbox{for all}~f\in L^p((0,\pi),v\,d\t),$$
if and only if
\begin{equation}\label{eq3}
\frac{w(I)^{1/q}\bar{v}(I)^{(p-1)/p}}{|I|^{1-\s}}\le C<\infty,\quad\hbox{for every interval}~I\subset(0,\pi).
\end{equation}
Here, for a set $A$ and a weight $u$, we write $\displaystyle u(A)=\int_Au(\t)\,d\t$.
\end{thm}

The weights we are going to need are
\begin{equation}\label{peso w}
w(\t)=(\sin\tfrac{\t}{2})^{(2\a+1)(1-q/2)}(\cos\tfrac{\t}{2})^{(2\b+1)(1-q/2)},
\end{equation}
and
\begin{equation}\label{peso v}
v(\t)=(\sin\tfrac{\t}{2})^{(2\a+1)(1-p/2)}(\cos\tfrac{\t}{2})^{(2\b+1)(1-p/2)},
\end{equation}
where $p,q,\a,\b$ are as in \eqref{chu 1} and \eqref{chu 2} of Theorem \ref{Thm:Lp-Lq}. Let us check that these weights satisfy the hypotheses of Theorem \ref{Thm:BS}. First we have to see that $w$ and $\bar{v}=v^{-1/(p-1)}$ are in $A_\infty$. Next, condition \eqref{eq3} will follow from the relation between the exponents $p$ and $q$ and the parameters $\a$ and $\b$ that we established. In order to do all these computations we will need the following two elementary lemmas.

\begin{lem}\label{Lem:Mario Bros}
Let $\lambda>-1$, $0\le a<b<\infty$. Then, $\displaystyle \min\{\lambda+1,1\}\le\frac{b^{\lambda+1}-a^{\lambda+1}}{b^{\lambda}(b-a)}\le \max\{\lambda+1,1\}$.
\end{lem}

\begin{proof}
We consider first $\lambda+1\ge1$. Then, as $0\le a/b<1$, we have $0\le(a/b)^{\lambda+1}\le a/b<1$, which implies that
$$\frac{b^{\lambda+1}-a^{\lambda+1}}{b^{\lambda}(b-a)}=\frac{1-(a/b)^{\lambda+1}}{1-a/b}\ge1.$$
On the other hand, applying the mean value theorem to the function $x^{\lambda+1}$, it follows that
\[
\frac{b^{\lambda+1}-a^{\lambda+1}}{b^{\lambda}(b-a)}=\frac{1-(a/b)^{\lambda+1}}{1-a/b}=(\lambda+1)\xi^{\lambda}\le \lambda+1
\]
for some $\xi\in(a/b,1)$. In the case $0<\lambda+1<1$, the proof follows analogous reasonings.
\end{proof}

\begin{lem}\label{Lem:Super Mario Bros}
Let $\lambda,\nu>-1$ and $I=(a,b)\subseteq (0,\pi)$. Then
$$\int_I \t^{\lambda}(\pi-\t)^{\nu}\,d\t\sim |I||I_0|^{\lambda}|I_{\pi}|^\nu,$$
where, for $\xi\in(0,\pi)$, $I_\xi$ denotes the least interval in $(0,\pi)$ that contains $I$ and $\xi$. Moreover,  the symbol $\sim$ means that there exist constants $0<c\leq C<\infty$, depending only on $\lambda$ and $\nu$, such that
\[
c\le\frac{1}{|I||I_0|^{\lambda}|I_{\pi}|^\nu}\int_I \t^{\lambda}(\pi-\t)^{\nu}\,d\t\le C.
\]
\end{lem}

\begin{proof}
We consider three cases. First assume that $0\le a<b\le 2\pi/3$. By Lemma \ref{Lem:Mario Bros}, we have
\[
\int_a^b\t^\lambda(\pi-\t)^{\nu}\,d\t\sim\int_a^b\t^\lambda\,d\t\sim b^{\lambda+1}-a^{\lambda+1}\sim b^{\lambda}(b-a).
\]
Analogously, for the case $\pi/3\le a<b\le\pi$,
$$\int_a^b\t^\lambda(\pi-\t)^{\nu}\,d\t\sim\int_a^b(\pi-\t)^\nu\,d\t\sim (\pi-a)^{\nu+1}-(\pi-b)^{\lambda+1}
\sim (\pi-a)^{\nu}(b-a).$$
Finally, when $0\le a<\pi/3<2\pi/3<b\le\pi$, by taking into account that $[\pi/3,\pi/3]\subseteq I\subseteq(0,\pi)$,
\[
\int_a^b\t^\lambda(\pi-\t)^{\nu}\,d\t\sim1 \sim(b-a)b^{\lambda}(\pi-a)^{\nu}.
\]
\end{proof}

We are ready to check that the weights $w$ in \eqref{peso w} and $v$ in \eqref{peso v} satisfy the hypotheses of Theorem \ref{Thm:BS}. Recall that for $\t\in(0,\pi)$, $\sin\frac{\t}{2}$ is equivalent to $\t$, and $\cos\frac{\t}{2}$ is equivalent to $(\pi-\t)$. By taking into account \eqref{chu 1} and Lemma \ref{Lem:Super Mario Bros} it is easy to see that $w\in A_q$ and $\bar{v}=v^{-1/(p-1)}\in A_{p/(p-1)}$. Therefore $w,\bar{v}\in A_\infty$. For \eqref{eq3} it is enough to see that
\begin{align*}
    P &:= \frac{1}{|I|^{1-\sigma}}\left(\int_I\t^{(2\a+1)(1-q/2)}(\pi-\t)^{(2\b+1)(1-q/2 )}\,d\t\right)^{1/q} \\
     &\qquad \times \left(\int_I\t^{(2\a+1)(1-p/2)\left(\frac{-1}{p-1}\right)}(\pi-\t)^{(2\b+1)(1-p/2) \left(\frac{-1}{p-1}\right)}\,d\t\right)^{(p-1)/p}\leq C<\infty,
\end{align*}
for every interval $I\subset(0,\pi)$. First, note that by \eqref{chu 1}, the exponents of $\t$ and $(\pi-\t)$ in both integrals are larger than $-1$. Therefore, Lemma \ref{Lem:Super Mario Bros} can be applied. Moreover, \eqref{chu 2} implies $1/q-1/p+\s\ge0$. Hence,
\begin{align*}
    P &\sim |I|^{\frac{1}{q}-\frac{1}{p}+\sigma}|I_0|^{(2\a+1)(1-\frac{q}{2})\frac{1}{q}+(2\a+1)(1-\frac{p}{2})\left(\frac{-1}{p}\right)} |I_{\pi}|^{(2\b+1)(1-\frac{q}{2})\frac{1}{q}+(2\b+1)(1-\frac{p}{2})\left(\frac{-1}{p}\right)} \\
     &= \frac{(\pi(b-a))^{\frac{1}{q}-\frac{1}{p}+\s}}{\pi^{\frac{1}{q}-\frac{1}{p}+\sigma}} |I_0|^{(2\a+1)(1-\frac{q}{2})\frac{1}{q}+(2\a+1)(1-\frac{p}{2})\left(\frac{-1}{p}\right)} |I_{\pi}|^{(2\b+1)(1-\frac{q}{2})\frac{1}{q}+(2\b+1)(1-\frac{p}{2})\left(\frac{-1}{p}\right)} \\
     &\le C|I_0|^{\frac{1}{q}-\frac{1}{p}+\sigma+(2\a+1)(1-\frac{q}{2})\frac{1}{q}+(2\a+1)(1-\frac{p}{2})\left(\frac{-1}{p}\right)} |I_{\pi}|^{\frac{1}{q}-\frac{1}{p}+\sigma+(2\b+1)(1-\frac{q}{2})\frac{1}{q}+(2\b+1)(1-\frac{p}{2})\left(\frac{-1}{p}\right)},
\end{align*}
where in the last inequality we used that $(\pi(b-a))^{\delta}\le (b(\pi-a))^{\delta}$ for $\delta=1/q-1/p+\s\ge0$. By \eqref{chu 2}, the last quantity above is uniformly bounded by a constant independent of $I$.

Thus, by applying Theorem \ref{Thm:BS} with $w$ as in \eqref{peso w}, $v$ as in \eqref{peso v} and $p,q,\a,\b$ as in \eqref{chu 1}-\eqref{chu 2}, we get that there exists a constant $C$ depending only on $\s$, $\a$ and $\b$ such that
\begin{equation}\label{weighted}
\|\mathcal{I}_\s g\|_{L^q((0,\pi),w\,d\t)}\leq C\|g\|_{L^p((0,\pi),v\,d\t)},
\end{equation}
for all $g\in L^p((0,\pi),v\,d\t)$. Recall now the following result about vector-valued extensions of bounded operators proved by J. Marcinkiewicz and A. Zygmund in \cite{Marcinkiewicz-Zygmund}.

\begin{lem}[Marcinkiewicz--Zygmund]\label{Lem:Marcin}
Consider $L^p=L^p(X,m)$, where $(X,m)$ is a $\sigma$-finite measure space. Let $T:L^p\to L^q$ be a bounded linear operator, $0<p\leq q<\infty$, with norm $\|T\|$. Then $T$ has an $\ell^2$--valued extension and
$$\Big\|\Big(\sum_j|Tf_j|^2\Big)^{1/2}\Big\|_{L^q}\leq\|T\|\Big\|\Big(\sum_j|f_j|^2\Big)^{1/2}\Big\|_{L^p},\quad f_j\in L^p.$$
\end{lem}

By applying Lemma \ref{Lem:Marcin} with $T=\mathcal{I}_\s$, since \eqref{weighted} holds, there exists a constant $C$ depending only on $\a$, $\b$ and $\s$ such that
\begin{equation}\label{weighted vectorial}
\Big\|\Big(\sum_{j=0}^\infty|\mathcal{I}_\s g_j|^2\Big)^{1/2}\Big\|_{L^q((0,\pi),w\,d\t)}\leq C\Big\|\Big(\sum_{j=0}^\infty |g_j|^2\Big)^{1/2}\Big\|_{L^p((0,\pi),v\,d\t)},
\end{equation}
for $g_j\in L^p((0,\pi),v\,d\t)$.

\subsubsection{Conclusion of the proof of Theorem \ref{Thm:Lp-Lq}}

We first observe that, by Theorem \ref{Thm:Jacobi kernel}, there exists a constant $C$ depending only on $\sigma$ and $\b$ such that for a nonnegative function $f$,
\begin{equation*}
\begin{aligned}
    u_j(\t)(\J^{\a+aj,\b+bj}&)^{-\s/2}(u_j^{-1}f)(\t) = u_j(\t)\int_0^\pi\K_\s^{\a+aj,\b+bj}(\t,\v)u_j^{-1}(\v)f(\v)~d\mu_{\a+aj,\b+bj}(\v) \\
     &\leq \frac{C}{(\sin\frac{\t}{2})^{\a+1/2}(\cos\frac{\t}{2})^{\b+1/2}} \int_0^\pi\frac{(\sin\frac{\v}{2})^{\a+1/2}(\cos\frac{\v}{2})^{\b+1/2}f(\v)}{|\t-\v|^{1-\s}}\,d\v \\
     &= \frac{C}{(\sin\frac{\t}{2})^{\a+1/2}(\cos\frac{\t}{2})^{\b+1/2}}\,\mathcal{I}_\s g(\t),
\end{aligned}
\end{equation*}
where
$$g(\t):=(\sin\tfrac{\t}{2})^{\a+1/2}(\cos\tfrac{\t}{2})^{\b+1/2}f(\t).$$
Therefore, from \eqref{weighted vectorial},
\begin{align*}
    \Big\|\Big(\sum_{j=0}^\infty&|u_j(\mathcal{J}^{\a+aj,\b+bj})^{-\s/2}(u_j^{-1}f_j)|^2\Big)^{1/2}\Big\|_{L^q((0,\pi),d\mu_{\a,\b})} \\
    &\leq C\Big\|\Big(\sum_{j=0}^\infty\frac{1}{(\sin\frac{\t}{2})^{2\a+1}(\cos\frac{\t}{2})^{2\b+1}}\,|\mathcal{I}_\s g_j(\t)|^2\Big)^{1/2}\Big\|_{L^q((0,\pi),d\mu_{\a,\b})} \\
    &= C\Big\|\Big(\sum_{j=0}^\infty|\mathcal{I}_\s g_j(\t)|^2\Big)^{1/2}\Big\|_{L^q((0,\pi),w\,d\t)} \\
    &\leq C\Big\|\Big(\sum_{j=0}^\infty |g_j|^2\Big)^{1/2}\Big\|_{L^p((0,\pi),v\,d\t)}= C\Big\|\Big(\sum_{j=0}^\infty |f_j|^2\Big)^{1/2}\Big\|_{L^p(d\mu_{\a,\b})},
\end{align*}
as desired.

\section{Fractional integrals on compact Riemannian symmetric spaces of rank one}\label{Section:Riemannian}

In this section we define the fractional integrals $(-\Delta_M)^{-\sigma/2}$ and the mixed norm spaces $L^p(L^2(M))$ on compact Riemannian symmetric spaces of rank one $M$ and we prove Theorems \ref{Thm:esfera} and \ref{Thm:proyectivos}.

\subsection{The unit sphere $\S^d$}\label{subsection:redonda}

Let $\S^d=\set{x\in\Real^{d+1}:x_1^2+\cdots+x_{d+1}^2=1}$ be the unit sphere in $\Real^{d+1}$, $d\geq2$. Let $\tilde{\Delta}_{\S^d}$ be the Laplace--Beltrami operator on $\S^d$. We set
$$-\Delta_{\S^d}:=-\tilde{\Delta}_{\S^d}+\left(\frac{d-1}{2}\right)^2.$$
It is well-known that $L^2(\S^d)=\bigoplus_{n=0}^\infty\mathcal{H}_n(\S^d)$, where $\mathcal{H}_n(\S^d)$ is the set of spherical harmonics of degree $n$ in $d+1$ variables, see \cite{BGM, Dunkl-Xu}. Each space $\mathcal{H}_n(\S^d)$ is an eigenspace of $\tilde{\Delta}_{\S^d}$ with eigenvalue $-n(n+d-1)$. Hence, $-\Delta_{\S^d}\mathcal{H}_n(\S^d)=(n+\frac{d-1}{2})^2\mathcal{H}_n(\S^d)$ and we can define, for $0<\sigma<1$, the fractional integral operator as
\begin{equation}\label{definicion integral fraccionaria}
(-\Delta_{\S^d})^{-\sigma/2}f=\sum_{n=0}^\infty\frac{1}{(n+\frac{d-1}{2})^\sigma}\operatorname{Proj}_{\mathcal{H}_n(\S^d)}(f),
\end{equation}
where $\operatorname{Proj}_{\mathcal{H}_n(\S^d)}(f)$ denotes the projection of $f\in L^2(\S^d)$ onto the eigenspace $\mathcal{H}_n(\S^d)$.

Now we define the mixed norm space  $L^p(L^2(\S^d))$. To this end, we need the following coordinates on $\S^d$, known as \textit{geodesic polar coordinates}. Each point on the sphere can be represented as
\begin{equation}\label{geodesicas}
\Phi(\theta,x')=(\cos\theta,x_1'\sin\theta,\ldots,x_d'\sin\theta)\in\S^d,
\end{equation}
for $\theta\in[0,\pi]$ and $x'\in\S^{d-1}$. We have
$$\int_{\S^d}f(x)\,dx=\int_0^\pi\int_{\S^{d-1}}f(\Phi(\theta,x'))\,dx'\,(\sin\theta)^{d-1}\,d\theta,$$
see \cite[p.~57]{BGM}. The mixed norm space $L^p(L^2(\S^d))$, $1\leq p<\infty$, is the set of functions $f$ on $\S^d$ such that
$$\|f\|_{L^p(L^2(\S^d))}:=\Bigg(\int_0^\pi\Bigg(\int_{\S^{d-1}}|f(\Phi(\theta,x'))|^2 dx'\Bigg)^{p/2}(\sin\theta)^{d-1}\,d\theta\Bigg)^{1/p}<\infty.$$
The spaces $L^p(L^2(\S^d))$ under this norm are Banach spaces.

By using the coordinates in \eqref{geodesicas}, it can be checked that
\begin{equation}\label{laplaciano polares}
-\Delta_{\S^d}=-\frac{\partial^2}{\partial\theta^2}-(d-1)\cot\theta\frac{\partial}{\partial\theta}+\left(\frac{d-1}{2}\right)^2 -\frac{1}{(\sin\theta)^2}\tilde\Delta_{\S^{d-1}},
\end{equation}
where $\tilde{\Delta}_\S^{d-1}$ is the spherical part of the Laplacian on $\Real^d$, acting on functions on $\Real^{d+1}$ by holding the first coordinate fixed and differentiating with respect to the remaining variables (see \cite[Section~3]{Sherman}). By using the description of $\mathcal{H}_n(\S^d)$ via spherical harmonics and the coordinates in \eqref{geodesicas}, it is an exercise to see that an orthonormal basis associated to \eqref{laplaciano polares} is given by
$$\varphi_{n,j,k}(x)=\psi_{n,j}(\theta)Y_{j,k}^d(x'),\quad x=\Phi(\theta,x'),~0<\theta<\pi,~x'\in\S^{d-1}.$$
See also \cite[Section~3]{Sherman}. Here, for $n\geq0$ and $j=0,1,\ldots,n$,
$$\psi_{n,j}(\theta)=a_{n,j}(\sin\theta)^jC_{n-j}^{j+\frac{d-1}{2}}(\cos\theta),$$
where $C_k^\lambda(x)$ is an ultraspherical polynomial as in \eqref{Gegenbauer}, and $a_{n,j}$ is a normalizing constant as follows
\begin{align*}
    a_{n,j} &= \|(\sin\theta)^jC_{n-j}^{j+\frac{d-1}{2}}(\cos\theta)\|_{L^2((0,\pi),(\sin\theta)^{d-1}\,d\theta)}^{-1} \\
     &= \frac{2^{-(j+\frac{d-1}{2})}\Gamma(2j+d-1)\Gamma(n+d/2)}{\Gamma(j+d/2)\Gamma(n+j+d-1)}\,d_{n-j}^{j+\frac{d-2}{2},j+\frac{d-2}{2}},
\end{align*}
with $d^{\a,\b}_n$ as in \eqref{que}. The functions $Y_{j,k}^d$, $k=1,\ldots,d(j)$, where $d(j)=(2j+d-2)\frac{(j+d-3)!}{j!(d-2)!}$, form an orthonormal basis of spherical harmonics on $\S^{d-1}$ of degree $j\geq0$. The orthogonal projections of $f$ onto the spaces $\mathcal{H}_n(\S^d)$ can be written as
\begin{equation}\label{proyecciones}
\operatorname{Proj}_{\mathcal{H}_n(\S^d)}(f)=\sum_{j=0}^n\sum_{k=1}^{d(j)}\langle f,\varphi_{n,j,k}\rangle_{L^2(\S^d)}\varphi_{n,j,k}.
\end{equation}

\begin{proof}[Proof of Theorem \ref{Thm:esfera} for $M=\S^d$]
Let $f$ be a function in $L^p(L^2(\S^d))$ and set $F:=f\circ\Phi$, where $\Phi$ is as in \eqref{geodesicas}. We can write
\begin{equation}\label{geografia}
F(\theta,x')=\sum_{j=0}^\infty\sum_{k=1}^{d(j)}F_{j,k}(\theta)Y_{j,k}^d(x'),\quad\hbox{for}~\theta\in[0,\pi],~x'\in\mathbb{S}^{d-1},
\end{equation}
where
$$F_{j,k}(\theta)=\int_{\mathbb{S}^{d-1}}F(\theta,x')\overline{Y_{j,k}^d(x')}\,dx'.$$
With this,
$$\|f\|_{L^p(L^2(\S^d))}=\Bigg(\int_0^\pi\Big(\sum_{j=0}^\infty\sum_{k=1}^{d(j)}|F_{j,k}(\theta)|^2\Big)^{p/2} (\sin\theta)^{d-1}d\theta\Bigg)^{1/p}.$$
By applying \eqref{geografia} and the orthogonality of the spherical harmonics $Y^d_{j,k}$, we can write
\begin{align}\label{coeficientes radiales theta}
    \langle f,\varphi_{n,j,k}\rangle_{L^2(\S^d)} &= \langle F,\psi_{n,j}Y^d_{j,k}\rangle_{L^2((0,\pi)\times\S^{d-1},(\sin\theta)^{d-1}d\theta\,dx')} \nonumber\\ &= \int_0^\pi\int_{\S^{d-1}}\sum_{l=0}^\infty\sum_{s=1}^{d(l)}F_{l,s}(\theta)Y^d_{l,s}(x')\psi_{n,j}(\t) \overline{Y^d_{j,k}(x')}\,dx'\,(\sin\theta)^{d-1}d\theta\nonumber \\
    &= \sum_{l=0}^\infty\sum_{s=1}^{d(l)}\int_0^\pi F_{l,s}(\theta)\psi_{n,j}(\t)(\sin\theta)^{d-1}d\theta\int_{\S^{d-1}}Y^d_{l,s}(x') \overline{Y^d_{j,k}(x')}\,dx' \nonumber\\
    &= \int_0^\pi F_{j,k}(\theta)\psi_{n,j}(\t)(\sin\theta)^{d-1}d\theta= \langle F_{j,k},\psi_{n,j}\rangle_{L^2((0,\pi),(\sin\theta)^{d-1}d\theta)}.
\end{align}
By using \eqref{definicion integral fraccionaria}, \eqref{proyecciones}, \eqref{coeficientes radiales theta} and algebraic manipulations in the sums,
\begin{align*}
    (-\Delta_{\S^d})^{-\sigma/2}f(x) &= \sum_{j=0}^\infty\sum_{n=j}^\infty\sum_{k=1}^{d(j)}\frac{\langle F_{j,k},\psi_{n,j}\rangle_{L^2((0,\pi),(\sin\theta)^{d-1}d\theta)}}{(n+\frac{d-1}{2})^\sigma}\psi_{n,j}(\theta)Y^d_{j,k}(x') \\
    &= \sum_{j=0}^\infty\sum_{n=0}^\infty\sum_{k=1}^{d(j)}\frac{\langle F_{j,k},\psi_{n+j,j}\rangle_{L^2((0,\pi),(\sin\theta)^{d-1}d\theta)}}{(n+\frac{2j+d-1}{2})^\sigma}\psi_{n+j,j}(\theta)Y^d_{j,k}(x') \\
    &= \sum_{j=0}^\infty\sum_{k=1}^{d(j)}\left[\sum_{n=0}^\infty\frac{\langle F_{j,k},\psi_{n+j,j}\rangle_{L^2((0,\pi),(\sin\theta)^{d-1}d\theta)}}{(n+\frac{2j+d-1}{2})^\sigma}\psi_{n+j,j}(\theta)\right]Y^d_{j,k}(x'). \\
\end{align*}
From here, by taking into account the definition of $\psi_{n,j}$, the relations \eqref{Gegenbauer} and \eqref{trig pol} and Theorem \ref{Thm:Lp-Lq} with $\a=\b=\frac{d-2}{2}$ and $a=b=1$, we get
\begin{align*}
    & \|(-\Delta_{\S^d})^{-\sigma/2}f\|_{L^q(L^2(\S^d))}^q \\
    &= \int_0^\pi\left(\sum_{j=0}^\infty\sum_{k=1}^{d(j)}\left[\sum_{n=0}^\infty\frac{\langle F_{j,k},\psi_{n+j,j}\rangle_{L^2((0,\pi),(\sin\theta)^{d-1}d\theta)}}{(n+\frac{2j+d-1}{2})^\sigma} \psi_{n+j,j}(\theta)\right]^2\right)^{q/2}(\sin\theta)^{d-1}d\theta \\
    &= \int_0^\pi\Bigg(\sum_{j=0}^\infty\sum_{k=1}^{d(j)}\Bigg[\sum_{n=0}^\infty \frac{\langle (\sin\theta)^{-j}F_{j,k},\P_n^{(\a+j,\a+j)} \rangle_{L^2(d\mu_{\a,\a})}}{(n+\frac{2j+d-1}{2})^\sigma}(\sin\theta)^j\P_n^{(\a+j,\a+j)}(\t)\Bigg]^2\Bigg)^{q/2}(\sin\theta)^{d-1}d\theta \\
     &= \int_0^\pi\Bigg(\sum_{j=0}^\infty\sum_{k=1}^{d(j)}\Big[(\sin\theta)^j (\mathcal{J}^{\a+j,\a+j})^{-\sigma/2}\left((\sin\phi)^{-j}F_{j,k}\right)(\t)\Big]^2\Bigg)^{q/2}(\sin\theta)^{d-1}d\theta \\
     &\leq C\Bigg(\int_0^\pi\Big(\sum_{j=0}^\infty \sum_{k=1}^{d(j)}|F_{j,k}(\theta)|^2\Big)^{p/2}(\sin\theta)^{d-1}d\theta\Bigg)^{q/p}= C\|f\|_{L^p(L^2(\S^{d}))}^q.
\end{align*}
The result is proved.
\end{proof}

\subsection{The real projective space $P_d(\Real)$}

To deal with the real projective space we just have to consider even functions on the sphere (as done in the end of Section 2 of \cite{Sherman}). More precisely, a function defined on $P_d(\Real)$ can be identified with an even function on $\S^d\subset\Real^{d+1}$. Indeed, $s\longmapsto\pm s$ is the projection from $\S^d$ to $P_d(\Real)$. Since the spherical harmonics in $\S^d$ satisfy $Y_{n,j}^{d+1}(-x)=(-1)^nY^{d+1}_{n,j}(x)$, the space $L^2_{\mathrm{e}}(\S^d)$ of even functions in $L^2(\S^d)$ can be decomposed as $L^2_{\mathrm{e}}(\S^d)=\bigoplus_{n=0}^\infty\mathcal{H}_{2n}(\S^d)$. Hence, the fractional integral on $P_d(\Real)$ is given by
\begin{align*}
    (-\Delta_{P_d(\Real)})^{-\sigma/2}f(x) &= \sum_{n=0}^\infty\frac{1}{(2n+\frac{d-1}{2})^\sigma}\,\operatorname{Proj}_{\mathcal{H}_{2n}(\S^d)}(f)(x) \\
    &= \sum_{n=0}^\infty\frac{1}{(2n+\frac{d-1}{2})^\sigma}\sum_{j=1}^{\delta(n)}\langle f,Y_{2n,j}^{d+1}\rangle_{L^2(\S^d)}Y_{2n,j}^{d+1}(x),
\end{align*}
with $\delta(n)=(4n+d-1)\frac{(2n+d-2)!}{(2n)!(d-1)!}$. By using the coordinates introduced in the previous subsection \eqref{geodesicas}, we can obtain the following orthonormal basis of $L^2(P_d(\Real))$:
$$\varphi_{2n,j,k}(x)=\psi_{2n,j}(\theta)Y_{j,k}^d(x'),\quad x=\Phi(\theta,x'),~0<\theta<\pi,~x'\in\S^{d-1}.$$
Thus, Theorem \ref{Thm:esfera} for $P_d(\Real)$ is then established directly from the sphere case.

\subsection{The projective spaces $\CP$, $\HP$ and $\Ca$}

In this subsection we take some results from \cite[Section~4]{Sherman}. Let $M$ be any of the projective spaces $\CP$, $\HP$ or $\Ca$. Let $\tilde{\Delta}_M$ be the Laplace--Beltrami operator on $M$. We set
$$-\Delta_M:=-\tilde{\Delta}_M+\left(\frac{m+d}{2}\right)^2,$$
where $d=2,4,8$, $m=l-2,2l-3,3$, for $\CP$, $\HP$ and $\Ca$, respectively. We have the orthogonal direct sum decomposition $L^2(M)=\bigoplus_{n=0}^\infty\mathcal{H}_n(M)$. Each space $\mathcal{H}_n(M)$ is finite-dimensional and corresponds to the eigenspace of $\tilde{\Delta}_M$ with respect to the eigenvalue $-n(n+m+d)$. From here it is readily seen that $\mathcal{H}_n(M)=\{f\in C^\infty(M):-\Delta_Mf=(n+\frac{m+d}{2})^2f\}$. Denoting by $\operatorname{Proj}_{\mathcal{H}_n(M)}(f)$ the orthogonal projection of $f\in L^2(M)$ onto $\mathcal{H}_n(M)$, we can define, for $0<\sigma<1$, the fractional integral operator as $$(-\Delta_M)^{-\sigma/2}f=\sum_{n=0}^\infty\frac{1}{(n+\frac{m+d}{2})^\sigma}\,\operatorname{Proj}_{\mathcal{H}_n(M)}(f).$$

Let $\B^{d+1}$ be the unit ball in $\Real^{d+1}$ and
$$\omega(r):=c_\omega r^{-1}(1-r)^m,\quad 0<r<1,$$
with $c_\omega=\frac{\Gamma(m+d+1)}{\Gamma(d)\Gamma(m+1)}$.  According to \cite[Lemma~4.12]{Sherman} there is a bounded linear map $E:L^1(M)\to L^1(\B^{d+1},\omega(|x|)\,dx)$ such that for every $f\in L^1(M)$,
$$\int_Mf\,d\mu_M=\int_{\B^{d+1}}E(f)\omega(|x|)\,dx,$$
where $d\mu_M$ is the Riemannian measure on $M$. This allows us to define the mixed norm spaces $L^p(L^2(M))$, $1\leq p<\infty$, as the set of functions $f$ on $M$ for which the norm
$$\|f\|_{L^p(L^2(M))}=\Bigg(\int_0^1\Bigg(\int_{\S^d}|E(f)(rx')|^2 \,dx'\Bigg)^{p/2}~\omega(r)r^d\,dr\Bigg)^{1/p}$$
is finite. The spaces $L^p(L^2(M))$ are Banach spaces.

It is shown in \cite[Corollary~4.26]{Sherman} that $E(\mathcal{H}_n(M))=\mathcal{H}_n(\B^{d+1},\omega)$. Here, the set $\mathcal{H}_n(\B^{d+1},\omega)$ is the orthocompliment of the space $\mathcal{F}_{n-1}(\B^{d+1})$ of functions on $\B^{d+1}$ which are polynomials of degree less than or equal to $n-1$ in the variables $x$ and $|x|$, $x\in\B^{d+1}$, in $\mathcal{F}_n(\B^{d+1})$ with respect to the inner product on $L^2(\B^{d+1},\omega(|x|)\,dx)$. For $r=|x|$ we let $\partial/\partial r$ the radial derivative and $\tilde{\Delta}_{\S^d}$ the Laplace--Beltrami operator on the sphere $\S^d$. The spaces $\mathcal{H}_n(\B^{d+1},\omega)$ are eigenspaces of the differential operator
$$\Lambda=r(r-1)\frac{\partial^2}{\partial r^2}+\left((m+d+1)r-d\right)\frac{\partial}{\partial r}+\left(\frac{m+d}{2}\right)^2-\frac{1}{r}\tilde{\Delta}_{\S^d},$$
corresponding to the eigenvalue $\left(n+\frac{m+d}{2}\right)^2$, see \cite[Theorem~4.22]{Sherman}. In this way it is possible to write $L^2(\B^{d+1},\omega(|x|)\,dx)=\bigoplus_{n=0}^\infty\mathcal{H}_n(\B^{d+1},\omega)$. Moreover, $\mathcal{H}_n(\B^{d+1},\omega)$ is the orthogonal direct sum of subspaces $\mathcal{H}_{n,j}(\B^{d+1},\omega)$, $0\leq j\leq n$. A basis of $\mathcal{H}_{n,j}(\B^{d+1},\omega)$ is
$$\varphi^M_{n,j,k}(x)=\psi^M_{n,j}(r)Y^{d+1}_{j,k}(x'),\quad x=rx',~0<r<1,~x'\in\S^d,$$
where, for $P_n^{(\a,\b)}$ as in \eqref{eq1},
$$\psi^M_{n,j}(r)=a^M_{n,j}r^jP_{n-j}^{(m,2j+d-1)}(2r-1),$$
and the set $\{Y^{d+1}_{j,k}\}_{1\leq k\leq d(j)}$ is an orthonormal basis of the space of spherical harmonics of degree $j$, and $d(j):=(2j+d-1)\frac{(j+d-2)!}{j!(d-1)!}$. By virtue of \eqref{trig pol} and \eqref{que} the normalizing constant is
$$a^M_{n,j}=\|\varphi^M_{n,j,k}\|_{L^2(\B^{d+1},\omega(|x|)\,dx)}^{-1}=c_\omega^{-1/2}d_{n-j}^{m,2j+d-1}.$$
Thus,
\begin{equation}\label{eq44}
E\left(\operatorname{Proj}_{\mathcal{H}_n(M)}(f)\right)=\sum_{j=0}^n\sum_{k=1}^{d(j)}\langle E(f),\varphi_{n,j,k}^{M}\rangle_{L^2(\B^{d+1},\omega(|x|)dx)}\varphi_{n,j,k}^M.
\end{equation}

\begin{proof}[Proof of Theorem \ref{Thm:proyectivos}]
Given a function $f$ on $M$ we use the notation $F:=E(f)$. Parallel to the case of the sphere, we can write
$$F(x)=\sum_{j=0}^\infty\sum_{k=1}^{d(j)}F_{j,k}(r)Y^{d+1}_{j,k}(x'),\quad\hbox{where}~x=rx'\in\B^{d+1},~r\in (0,1),~x'\in\mathbb{S}^d,$$
and
$$F_{j,k}(r)=\int_{\mathbb{S}^{d}}F(rx')\overline{Y^{d+1}_{j,k}(x')}\,dx'.$$
Then
$$\|f\|_{L^p(L^2(M))}=\Bigg(\int_0^1\Big(\sum_{j=0}^\infty\sum_{k=1}^{d(j)}|F_{j,k}(r)|^2\Big)^{p/2} \omega(r)r^ddr\Bigg)^{1/p}.$$
Proceeding as in Subsection \ref{subsection:redonda},
$$\langle F,\varphi^M_{n,j,k}\rangle_{L^2(\B^{d+1},\omega(|x|)\,dx)}=\langle F_{j,k},\psi^M_{n,j}\rangle_{L^2((0,1),\omega(r)r^ddr)},$$
and then, by \eqref{eq44},
\begin{equation}\label{primera}
\begin{aligned}
    E\left((-\Delta_M)^{-\sigma/2}f\right)(x) &= \sum_{n=0}^\infty\sum_{j=0}^n\sum_{k=1}^{d(j)}\frac{\langle F,\varphi^M_{n,j,k}\rangle_{L^2(\B^{d+1},\omega(|x|)dx)}}{(n+\frac{m+d}{2})^\sigma}\,\varphi_{n,j,k}^M(x) \\
    &= \sum_{j=0}^\infty\sum_{k=1}^{d(j)}\left[\sum_{n=0}^\infty\frac{\langle F_{j,k},\psi^M_{n+j,j}\rangle_{L^2((0,1),\omega(r)r^ddr)}}{(n+\frac{2j+m+d}{2})^\sigma}\,\psi^M_{n+j,j}(r)\right]Y^{d+1}_{j,k}(x'). \\
\end{aligned}
\end{equation}
The change of variable $2r-1=\cos\theta$ implies
\begin{equation}\label{segunda}
\langle F_{j,k},\psi^M_{n+j,j}\rangle_{L^2((0,1),\omega(r)r^ddr)}= c_\omega^{1/2}\langle(\cos\tfrac{\theta}{2})^{-2j}G_{j,k},\P_n^{(m,2j+d-1)}\rangle_{L^2(d\mu_{m,2j+d-1})},
\end{equation}
where we have defined $G_{j,k}(\t):=F_{j,k}(\frac{1+\cos\t}{2})$. From \eqref{primera}, \eqref{segunda}, the change of variable $2r-1=\cos\theta$ and Theorem \ref{Thm:Lp-Lq} with $\a=m$, $\b=d-1$, $a=0$ and $b=2$, we obtain
\begin{align*}
    &\|(-\Delta_M)^{-\sigma/2}f\|_{L^q(L^2(M))}^q = \int_0^1\Bigg(\sum_{j=0}^\infty\sum_{k=1}^{d(j)}\Big[\sum_{n=0}^\infty\frac{\langle F_{j,k},\psi^M_{n+j,j}\rangle_{L^2((0,1),\omega(r)r^ddr)}}{(n+\frac{2j+m+d}{2})^\sigma}\,\psi^M_{n+j,j}(r)\Big]^2\Bigg)^{q/2}\omega(r)r^ddr \\
    &= c_\omega\int_0^\pi\Bigg(\sum_{j=0}^\infty\sum_{k=1}^{d(j)}\Big[\sum_{n=0}^\infty\frac{\langle (\cos\frac{\theta}{2})^{-2j}G_{j,k},\P_n^{(\a,\b+2j)}\rangle_{L^2(d\mu_{\a,\b+2j})}}{(n+\frac{2j+m+d}{2})^\sigma} (\cos\tfrac{\theta}{2})^{2j}\P_n^{(\a,\b+2j)}(\theta)\Big]^2\Bigg)^{q/2}d\mu_{\a,\b}(\theta) \\
    &= c_\omega\int_0^\pi\Bigg(\sum_{j=0}^\infty\sum_{k=1}^{d(j)}\left[(\cos\tfrac{\theta}{2})^{2j}(\J^{(\a,\b+2j)})^{-\sigma/2} \left((\cos\tfrac{\phi}{2})^{-2j}G_{j,k}\right)(\theta)\right]^2\Bigg)^{q/2}d\mu_{\a,\b}(\theta) \\
    &\leq C\Bigg(\int_0^\pi\Big(\sum_{j=0}^\infty \sum_{k=1}^{d(j)}|G_{j,k}(\t)|^2\Big)^{p/2}d\mu_{\a,\b}(\theta)\Bigg)^{q/p} \\
    &= C\Bigg(\int_0^1\Big(\sum_{j=0}^\infty \sum_{k=1}^{d(j)}|F_{j,k}(r)|^2\Big)^{p/2}\omega(r)r^ddr\Bigg)^{q/p}= C\|f\|_{L^p(L^2(M))}^q.
\end{align*}
\end{proof}

\begin{rem}
A basis of orthogonal polynomials in $L^2(\B^{d+1},W_\mu(x)dx)$, with the weight $W_\mu(x)=(1-|x|^2)^{\mu-1/2}$, for $\mu>-1/2$ and $x\in\B^{d+1}$, can be given in terms of products of spherical harmonics in $\S^d$ and Jacobi polynomials evaluated at $2|x|^2-1$, see \cite[Chapter~2]{Dunkl-Xu}. For this system there is an underlying linear second order differential operator $L_\mu$, whose negative powers $(L_\mu)^{-\sigma/2}$ can be considered. By defining in a natural way the mixed norm spaces on $\B^{d+1}$ with the measure $W_\mu(x)dx$, it is then possible to get $L^p(L^2(\B^{d+1},W_\mu(x)dx))-L^q(L^2(\B^{d+1},W_\mu(x)dx))$ estimates for $(L_\mu)^{-\sigma/2}$ in this general setting by just applying Theorem \ref{Thm:Lp-Lq} as above.
\end{rem}



\end{document}